\documentclass[fleqn]{zzz}
\usepackage{mathrsfs}
\usepackage{color}
\usepackage{hyperref}

\hypersetup{
   colorlinks = true,
   urlcolor = blue,
   linkcolor = red
}

\numberwithin{equation}{section} 

\newcommand{\hyp}[5]{{}_{#1}F_{#2}\!\left(\genfrac{}{}{0pt}{}{#3}{#4};#5\right)}
\newcommand{\Hyper}[5]{{}_{#1}F_{#2} \left(\begin{array}{c} {#3} \\ {#4}
\end{array};{#5} \right)}
\newcommand{\qhyp}[5]{\,\mbox{}_{#1}\phi_{#2}\!\left(
\genfrac{}{}{0pt}{}{#3}{#4};#5\right)}
\newcommand{\N}{{\mathbb N}}
\newcommand{\C}{{\mathbb C}}

\newtheorem{thm}[lemma]{Theorem}
\newtheorem{cor}[lemma]{Corollary}
\newtheorem{prop}[lemma]{Proposition}

\begin{document}

\renewcommand{\PaperNumber}{***}

\FirstPageHeading

\ShortArticleName{Connection relations by power collection:~Meixner}

\ArticleName{The power collection method for connection relations: Meixner polynomials}

\Author{Michael A.~Baeder,$^\ast$ Howard S.~Cohl,$^\dag$
Roberto S.~Costas-Santos,$^\ddag$
and Wenqing Xu$^\S$}

\AuthorNameForHeading{M.~A.~Baeder, H.~S.~Cohl, 
R. S.~Costas-Santos, W.~Xu}

\Address{$^\ast$~Michael Baeder, Institutional Clients Group, 
Citigroup Inc., New York, NY 10013, USA
} 
\EmailD{mabaeder@gmail.com} 

\Address{$^\dag$~Applied and Computational Mathematics Division,
National Institute of Standards and Technology,
Gaithersburg, MD 20899-8910, USA
\URLaddressD{http://www.nist.gov/itl/math/msg/howard-s-cohl.cfm}
} 
\EmailD{howard.cohl@nist.gov} 

\Address{$^\ddag$~Departamento de F\'isica y Matem\'{a}ticas,
Universidad de Alcal\'{a},
c.p. 28871, Alcal\'{a} de Henares, Madrid, Spain
} 
\URLaddressD{http://www.rscosan.com}
\EmailD{rscosa@gmail.com} 

\Address{$^\S$~Department of Mathematics and Statistics, California Institute
of Technology, CA 91125, USA
\EmailD{williamxuxu@yahoo.com} 
}

\ArticleDates{Received ???, in final form ????; Published online ????}

\Abstract{We introduce the power collection method for easily 
deriving connection relations 
for certain hypergeometric orthogonal polynomials in the $(q-)$Askey scheme.
We summarize the full-extent to which the power collection method may be used.
As an example, we use the power collection method to derive connection and 
connection-type relations for Meixner and Krawtchouk polynomials. These relations 
are then used to derive generalizations of generating functions for these 
orthogonal polynomials.  The coefficients of these generalized generating 
functions are in general, given in term of multiple hypergeometric functions.
From derived generalized generating functions, we derive corresponding contour 
integral and infinite series expressions by using orthogonality.
}

\Classification{33C45; 05A15; 33C20; 34L10; 30E20}

\section{Introduction}
\label{Introduction}
Orthogonal polynomials are a group of polynomial families such that 
any two different polynomials in that family are orthogonal to each other 
under some inner product.
This relation can sometimes be expressed discretely for a sequence of 
orthogonal polynomials.
For instance, given $\{P_n(x;{\bf a})\}$, $n\in{\mathbb N}_0$, with discrete weight 
$w_x\in\mathbb C$,  ${\bf a}$ is a set  of free parameters, and $r_n\in\mathbb C$, 
then one  may have the following  discrete orthogonality relation
\[
\sum_{x=0}^{\infty}P_m(x;{\bf a})P_n(x;{\bf a})w_x({\bf a})=r_n({\bf a})\delta_{m,n}.
\]
In this paper we discuss connection and connection-type relations, and 
generalizations  of generating functions from these relations for a family 
of discrete hypergeometric orthogonal  polynomials, namely the Meixner 
and Krawtchouk polynomials \cite[Sections 9.10-11]{Koekoeketal}.
Note that we use the terminology that a double connection relation is a connection 
relation with two free parameters, and a triple connection relation is a connection
relation with three free parameters.

The paper is organized as follows. In Section \ref{section2_preliminaries},
some mathematical preliminaries which are used in our proofs are introduced. 
In Section \ref{section3_power}, the power collection method for deriving
connection relations is explained. Polynomials in which one can apply the power
collection method are also listed.
In Section \ref{section3_Meixner}, connection and connection-type relations 
are given for Meixner and Krawtchouk polynomials.
In Section \ref{GeneralizedGeneratingFunctions}, generalizations of generating 
functions for Meixner and Krawtchouk polynomials are presented.
In Section \ref{ResultsUsingOrthogonality}, infinite series expressions are given 
which are derived using orthogonality for Meixner and Krawtchouk polynomials.

\section{Preliminaries:~hypergeometric functions}\label{section2_preliminaries}

Our generalizations of generating functions rely on Pochhammer symbols.
The Pochhammer symbol, also called the shifted factorial, is a special 
function that  is used to express coefficients of polynomials. 
They can be used to express binomial coefficients, coefficients of derivatives 
of polynomials, and are integral to the  definition of hypergeometric functions.
The Pochhammer symbol is defined for $a\in\mathbb C$, $n\in{\mathbb N}_0$, 
such that
\begin{equation}
\label{2:1}
(a)_n:=(a)(a+1)\cdots(a+n-1),
\end{equation}
where as have assumed (and throughout this paper) that the empty 
product is unity.
Define 
\begin{eqnarray*}
&&{\widehat{\mathbb C}}:=\{z\in\mathbb C: -z\not\in {\mathbb N}_0\},\\
&&{{\mathbb C}_{0}}:=\{z\in\mathbb C: z\ne 0\},\\
&&{{\mathbb C}_{0,1}}:=\{z\in\mathbb C: z\not\in\{0,1\}\}.
\end{eqnarray*}
One has the following useful identities for Pochhammer symbols, namely
for $n\in{\mathbb N}_0$,
\begin{eqnarray}
&&\label{pochsum}
(a)_n=\frac{\Gamma(a+n)}{\Gamma(a)}, \\
&& \label{negpoch}
\Gamma(a-n)=\frac{(-1)^n\Gamma(a)}{(-a+1)_n},
\end{eqnarray}
where $a\in{\widehat{\mathbb C}}$,
and for $k\in{\mathbb N}_0$, $a\in{\mathbb C}$, one has
\begin{equation}
\label{PochAdd}
(a)_{n+k} = (a)_n(a+n)_k=(a)_k(a+k)_n.
\end{equation}
Another useful identity which we use is that for $n,k\in\N_0$, then 
\[
(-n)_k=\frac{(-1)^kn!}{(n-k)!},
\]
if $0\le k\le n$, and zero otherwise.

Moreover, for many of the proofs in this paper, we will need the following 
inequalities for Pochhammer symbols
\cite[Lemma 12]{CohlMacKVolk}. Let $j\in \mathbb N$, $k,n\in \mathbb N_0$, 
$z\in \mathbb C$, $\Re u>0$, $w>-1$, $v\ge 0$. Then
\begin{eqnarray}
\label{2:7} && |(u)_j|\ge (\Re u) (j-1)!, \\
\label{2:8} && \dfrac{(v)_n}{n!}\le  (1+n)^v, \\
\label{2:9} && (n+w)_k\le \max\{1,2^w\}\frac{(n+k)!}{n!}, \\
\label{2:10} && (z+k)_{n-k}\le \frac{n!}{k!} (1+n)^{|z|}.
\end{eqnarray}
The generalized generating functions we present in this paper often have 
coefficients which can be expressed in terms of generalized hypergeometric 
functions. 
Generalized hypergeometric functions ${}_rF_s$ are special functions which 
can be represented by a hypergeometric series.  These are solutions of a 
$\max(s+1,r)$\textsuperscript{th} order differential equation 
with three regular singular points.
The generalized hypergeometric function is defined as
\cite [(1.4.1)]{Koekoeketal}
\begin{equation}
\label{2:5}
\hyp{r}{s}{a_1,\ldots,a_r}{b_1,\ldots,b_s}{z}
:=\sum_{k=0}^\infty
\frac{(a_1)_k\cdots(a_r)_k}
{(b_1)_k\cdots(b_s)_k}
\frac{z^k}{k!}.
\end{equation}
\noindent For instance, we often take advantage of the 
binomial theorem \cite[(1.5.1)]{Koekoeketal} which can be expressed as
\begin{equation}
\label{2:6}
\hyp10a-{z}=(1-z)^{-a},\qquad |z|<1.
\end{equation}
The $q$-Pochhammer symbol ($q$-shifted factorial) is defined for $ n \in \N_0 $ such that
\begin{equation}
\label{qPoch}
(a;q)_0:=1,
\quad
(a;q)_n:=(1-a)(1-aq)\cdots(1-aq^{n-1}),
\end{equation}
where $0<q<1$, $a \in \C $.

\noindent The basic hypergeometric series is defined as
\begin{equation}
{}_r\phi_s\left(
\begin{array}{c}
a_1,\ldots,a_r\\
b_1,\ldots,b_s
\end{array}
\Bigg|q,z
\right):=\sum_{k=0}^\infty
\frac{(a_1,\ldots,a_r;q)_k}
{(q,b_1,\ldots,b_s;q)_k}
\left((-1)^kq^{k\choose2}\right)^{1+s-r}
z^k.
\label{bhs}
\end{equation}
We have also taken advantage of the $q$-binomial theorem \cite[(1.11.1)]{Koekoeketal}
\[
\qhyp10a-{q,z}=\frac{(az;q)_\infty}{(z;q)_\infty}.
\]
Sometimes, the coefficients of our generalized generating funcitons are given in 
terms of double and triple hypergeometric functions.  There exists a large 
classification of such functions. The versions of these functions which we 
encounter are given as follows.  For double hypergeometric series we encounter
the function $F_1$ which is an Appell series. These are hypergeometric series 
in two variables and are defined as $\cite[(16.13.1)]{NIST}$
\begin{equation}
\label{3:14}
F_1\left(a,b,b^{'};c;x,y\right):=\sum_{m,n=0}^\infty\frac{(a)_{m+n} 
(b)_m(b^{'})_n}{(c)_{m+n}}\frac{x^m}{m!}\frac{y^n}{n!}.
\end{equation}
We also encounter the function $\Phi_2$, which is a Humbert hypergeometric 
series of two variables defined as \cite[p.~25]{SriKarl}
\begin{equation}
\label{kummer}
\Phi_2\left(\beta,\beta^{'};\gamma;x,y\right):=\sum_{m,n=0}^\infty
\frac{(\beta)_m(\beta^{'})_n}{(\gamma)_{m+n}} \frac{x^m}{m!}\frac{y^n}{n!}.
\end{equation}
The function $F_D^{(3)},$ a hypergeometric function of three-variables, 
is a form of the triple Lauricella series
defined as \cite[p.~33]{SriKarl}
\begin{equation}
\label{Appell_D}
F_D^{(3)}\left(a,b_1,b_2,b_3;c;x,y,z\right):=\sum_{m,n,p=0}^\infty
\frac{(a)_{m+n+p}\,(b_1)_{m}\,(b_2)_{n}\,(b_3)_{p}}{(c)_{m+n+p}}\,
\frac{x^{m}}{m!}\frac{y^{n}}{n!}\frac{z^{p}}{p!}.
\end{equation}
The function $\Phi_2^{(3)}$ is a confluent form of the triple Lauricella 
series defined as \cite[p.~34]{SriKarl}
\begin{equation}
\label{kummer_3}
\Phi_2^{(3)}\left(b_1,b_2,b_3;c;x,y,z\right):=\sum_{m,n,p=0}^\infty
\frac{(b_1)_{m}\,(b_2)_{n}\,(b_3)_{p}}{(c)_{m+n+p}}\,
\frac{x^{m}}{m!}\frac{y^{n}}{n!}\frac{z^{p}}{p!}.
\end{equation}

\section{The power collection method and its orthgonal polynomials}
\label{section3_power}

In this section, we describe what we refer to as the power collection method. 
This method can be used to easily derive connection relations for generalized
hypergeometric and basic hypergeometric orthogonal polynomials. The method
starts with a generating function such as
\begin{equation}
f(x,t;{\bf a})=\sum_{n=0}^{\infty}c_n({\bf a})P_n(x;{\bf a})t^n,
\label{genfuninthissection}
\end{equation}
for a hypergeometric orthogonal polynomial $P_n(x;{\bf a})$, where
${\bf a}$ is a set of arbitrary parameters, $x,t\in{\mathbb C}$, $|t|<1$.
For the power collection method to work, $f(x,t;{\bf a})$ must be in a 
particular elementary form, namely that it contains a simple ($q$-)binomial product
which can be expanded using the ($q$-)binomial theorem. It is also furthermore
crucial that if $\alpha\in{\bf a}$, then only the binomial term in the generating 
function may contain $\alpha$.  If this is the case, then the power collection
method may be used to easily derive a connection relation for the free parameter $\alpha$.
Generalized generating functions for hypergeometric orthogonal polynomials 
(see for instance \cite{BaedCohlVolk}) are produced by 
applying series rearragement to known generating functions using 
derived connection relations. Hence the power collection method is useful
for obtaining identities such as these.

In the context of the generating function (\ref{genfuninthissection}), consider 
free parameters $\alpha,\beta\in{\bf a}$, 
such that $\alpha,\beta$ are of the same type.  The power collection 
method\footnote{Note that this 
simple method was originally described to 
H.~S.~Cohl by Mourad Ismail. Ismail has also explained that this method is not new,
and has been used previously in the literature.}
proceeds by multiplying the ($q$-)binomial on the left-hand side of the 
generating function (\ref{genfuninthissection}) 
by a similar expression containing an alternate free 
parameter $\beta$ instead of $\alpha$. 
On the left-hand side, utilizing the binomial theorem and rearranging the nested 
series, produces the 
original generating function, however expressed in terms of $\beta$.  
If this method succeeds, by collecting terms corresponding to $t$ in the 
resulting expression, the coefficients of the expansion produces a connection 
relation in terms of the free parameters $\alpha,\beta$.

\vspace{0.2cm}
We now give an example of how the power collection method can be used for Meixner
polynomials to obtain the well known connection relation
\begin{equation}
\label{wellknownconnectionMeixner}
M_{n}(x;\alpha,c)
=\frac{1}{(\alpha)_n}\sum_{k=0}^{n}\binom{n}{k}(\alpha-\beta)_{n-k}(\beta)_k
M_k(x;\beta,c).
\end{equation}
In the following elementary generating function for Meixner 
polynomials \cite[(9.10.11)]{Koekoeketal},
the left-hand side is given in terms of a binomial expression, namely
\[
\label{3:12}
\left(1-\frac{t}{c}\right)^{x}(1-t)^{-x-\alpha}=
\sum_{n=0}^\infty\frac{(\alpha)_n}{n!}M_n(x;\alpha,c)t^n,
\]
where $|t|<|c|<1$.
Multiplying the left-hand side by
$\left( 1-t \right)^{-\beta}/\left( 1-t \right)^{-\beta}$,
and expressing it in terms of the original generating function, produces
\[
\left( 1-t \right) ^{\beta-\alpha}
\sum_{n=0}^\infty\frac{(\beta)_n}{n!}M_n(x;\beta,c)t^n
=
\sum_{n=0}^\infty\frac{(\alpha)_n}{n!}M_n(x;\alpha,c)t^n.
\]
Applying the binomial theorem (\ref{2:6}) to the above expression yields
\[
\sum_{k=0}^\infty \frac{(\alpha-\beta)_k}k{k!}t^k
\sum_{n=0}^\infty\frac{(\beta)_n}{n!}M_n(x;\beta,c)t^n
=
\sum_{n=0}^\infty\frac{(\alpha)_n}{n!}M_n(x;\alpha,c)t^n,
\]
and after collecting terms associated with $t$ produces
\[
\sum_{n=0}^\infty t^n
\sum_{k=0}^n \frac{(\alpha-\beta)_{n-k}(\beta)_k}{(n-k)!k!}M_k(x;\beta,c)
=
\sum_{n=0}^\infty t^n \frac{(\alpha)_n}{n!}M_n(x;\alpha,c).
\]
If we rearrange this expression, one produces
\[
\sum_{n=0}^\infty t^n
\left(\frac{(\alpha)_n}{n!}M_n(x;\alpha,c) 
-\sum_{k=0}^n \frac{(\alpha-\beta)_{n-k}(\beta)_k}{(n-k)!k!}M_k(x;\beta,c) \right)
=0.
\]
Since each term corresponding to $t^n$ in the above expression is linearly 
independent, the connection
relation (\ref{wellknownconnectionMeixner}) naturally follows.

\vspace{0.3cm}
This method is quite powerful and can be applied in many different contexts
of basic and generalized hypergeometric orthgonal polynomials.
Note that the continuous $q$-Hermite and discrete $q$-Hermite I \& II polynomials
are not displayed in the following list, even though these polynomials could 
potentially profit from use of the power collection method. 
The reason is that these polynomials contain no free parameters (other than $q$), 
and hence ordinary connection relations for these polynomials
(not in terms of $q$), do not exist.
Furthermore, for the continuous $q$-ultraspherical/Rogers polynomials, the 
generating function \cite[(14.10.27)]{Koekoeketal} may be used with
the power collection method to produce the connection relation for these
polynomials.  However, the connection relation for these polynomials 
is well known. For the connection relation, see for instance 
\cite[Section 13.3]{Ismail}, and for generalized generating functions
see \cite{CohlCostasHwang}.

We now provide a list of generalized and basic hypergeometric 
orthogonal polynomials in which one may apply the power 
collection method to easily obtain connection relations.
We also display the generating function for these polynomials, which is the 
main vehicle for the method to work.

\begin{itemize}
\item
Continuous dual Hahn polynomials. 
The relevant generating function is
\cite[(9.2.12)]{Koekoeketal}
\[
(1-t)^{-c+ix} \hyp21{a+ix,b+ix}{a+b}{t}
=\sum_{n=0}^{\infty}\frac{S_n(x^2;a,b,c)}{(a+b)_n n!}t^n.
\]
These polynomials have 3 free parameters and 5 known generating functions.
Note that the parameters $a$, $b$, and $c$ are symmetrical.
The power collection method will produce 1 connection relation for each symmetric free parameter.
Combining these connection relations produces 3 double connection relations and
one triple connection relation, for a total of 7 connection relations.

\item
Dual Hahn polynomials.
The relevant generating function is
\cite[(9.6.11)]{Koekoeketal}
\[
(1-t)^{N-x}\hyp21{-x,-x-\delta}{\gamma+1}{t}
=\sum_{n=0}^{N}\frac{(-N)_n}{n!}R_n(\lambda (x);\gamma,\delta,N) t^n,
\]
where $\lambda(x):=x(x+\gamma+\delta+1).$
These polynomials have 3 free parameters and 4 known generating functions.
The power collection method will produce 1 connection relation based on parameter $N$.
\item
Bessel polynomials.
The relevant generating function is
\cite[(9.13.10)]{Koekoeketal}
\[
(1-2xt)^{-\frac{1}{2}}
\left(\frac{2}{1+\sqrt{1-2xt}}\right)^a
\exp \left(\frac{2t}{1+\sqrt{1-2xt}}\right)
=\sum_{n=0}^\infty \frac{y_n(x;a)}{n!}t^n.
\]
These polynomials have 1 free parameter and 2 known generating functions.
The power collection method will produce 1 connection relation for the free parameter.
\item
Charlier polynomials.
The relevant generating function is
\cite[(9.14.11)]{Koekoeketal}
\[
e^t \left(1-\frac{t}{a}\right)^x
=\sum_{n=0}^\infty \frac{C_n(x;a)}{n!}t^n.
\]
These polynomials have 1 free parameter and 1 known generating function.
The power collection method will produce 1 connection relation 
for the free parameter, but no generalized generating functions
since the above generating function is the only known generating 
function for Charlier polynomials.
\item
Continuous dual $q$-Hahn polynomials.
The relevant generating function is
\cite[(14.7.11)]{Koekoeketal}
\[
\frac{(ct;q)_{\infty}}{(e^{i \theta}t;q)_{\infty}}
\qhyp21{ae^{i \theta},be^{i \theta}}{ab}{q,e^{-i\theta}t}
=\sum_{n=0}^\infty \frac{p_n(x;a,b,c|q)}{(ab,q;q)_n}t^n,
\]
where $x=\cos \theta$.
These polynomials have 3 free parameters and 4 known generating functions.
Note that parameters $a$, $b$, and $c$ are symmetrical.
The power collection method will produce 1 connection relation for each symmetric free parameter.
Combining these connection relations will produce 3 double connection relations and
one triple connection relation, for a total of 7 connection relations.
\item
Dual $q$-Hahn polynomials.
The relevant generating function is
\cite[(14.7.11)]{Koekoeketal}
\[
(q^{-N}t;q)_{N-x} \qhyp21{q^{-x},\delta^{-1}q^{-x}}{\gamma q}{q,\gamma \delta q^{x+1} t}
=\sum_{n=0}^{N}\frac{(q^{-N};q)_n}{(q;q)_n}R_n(\mu (x);\gamma,\delta,N|q) t^n.
\] 
where $\mu(x):=q^{-x}+\gamma\delta q^{x+1}$.
These polynomials have 3 free parameters and 2 known generating functions.
The power collection method will produce 1 connection relation based on parameter $N$.
\item
Al-Salam-Chihara polynomials.
The relevant generating function is
\cite[(14.8.13)]{Koekoeketal}
\[
\frac{(at,bt;q)_\infty}{(e^{i \theta}t,e^{-i \theta}t;q)_\infty}
=\sum_{n=0}^\infty \frac{Q_n(x;a,b|q)}{(q;q)_n} t^n,
\]
where $x=\cos \theta$.
These polynomials have 2 free parameters and 4 known generating functions.
Note that parameters $a$ and $b$ are symmetrical.
The power collection method will produce 1 connection relation for each free parameter. Combining
these connection relations will produce 1 double connection relation for a total of
3 connection relations.
\item
$q$-Meixner-Pollaczek polynomials.
The relevant generating function is
\cite[(14.9.11)]{Koekoeketal}
\[
\left|\frac{(ae^{i \phi}t;q)_\infty}{(e^{i(\theta + \phi)}t;q)_\infty}\right|
=\frac{(ae^{i \phi}t,ae^{-i \phi}t;q)_\infty}{(e^{i(\theta + \phi)}t,e^{-i(\theta + \phi)}t;q)_\infty}
=\sum_{n=0}^{\infty}P_n(x;a,\phi|q) t^n,
\]
where $x=\cos(\theta + \phi)$.
These polynomials have 2 free parameters and 2 known generating functions.
The power collection method will produce 1 connection relation for each free parameter. Combining
these connection relations will produce 1 double connection relation for a total of
3 connection relations.

\item
Big $q$-Laguerre polynomials.
The relevant generating function is
\cite[(14.11.11)]{Koekoeketal}
\[
(bqt;q)_\infty \qhyp21{aqx^{-1},0}{aq}{q,xt}
=\sum_{n=0}^\infty \frac{(bq;q)_n}{(q;q)_n}P_n(x;a,b;q) t^n.
\]
These polynomials have 2 free parameters and 3 known generating functions.
Note that parameters $a$ and $b$ are symmetrical.
The power collection method will produce 1 connection relation for each free parameter. Combining
these connection relations will produce 1 double connection relation for a total of
3 connection relations.
\item
Affine $q$-Krawtchouk polynomials.
The relevant generating function is
\cite[(14.16.11)]{Koekoeketal}
\[
(q^{-N}t;q)_{N-x}\qhyp11{q^{-x}}{pq}{q,pqt}
=\sum_{n=0}^N \frac{(q^{-N};q)_n}{(q;q)_n}K_n^{\rm Aff}(q^{-x};p,N;q)t^n.
\]
These polynomials have 2 free parameters and 2 known generating functions.
The power collection method will produce 1 connection relation based on parameter $N$.
\item
Dual $q$-Krawtchouk polynomials.
The relevant generating function is
\cite[(14.17.11)]{Koekoeketal}
\[
(cq^{-N}t;q)_x (q^{-N}t;q)_{N-x}
=\sum_{n=0}^N \frac{(q^{-N};q)_n}{(q;q)_n}K_n(\lambda(x);c,N|q) t^n,
\]
where $\lambda(x):=q^{-x}+cq^{x-N}.$
These polynomials have 2 free parameters and 1 known generating function.
The power collection method will produce 1 connection relation for each free parameter. Combining
these connection relations will produce 1 double connection relation for a total of
3 connection relations, but no generalized generating functions
since the above generating function is the only known generating 
function for dual $q$-Krawtchouk polynomials.
\item
Continuous big $q$-Hermite polynomials.
The relevant generating function is
\cite[(14.18.13)]{Koekoeketal}
\[
\frac{(at;q)_\infty}{(e^{i \theta}t,e^{-i \theta} t;q)_\infty}
=\sum_{n=0}^\infty \frac{H_n(x;a|q)}{(q;q)_n}t^n,
\]
where $x=\cos \theta$.
These polynomials have 1 free parameter and 3 known generating functions.
The power collection method will produce 1 connection relation for the free parameter.
\item
Al-Salam-Carlitz I polynomials.
The relevant generating function is
\cite[(14.24.11)]{Koekoeketal}
\[
\frac{(t,at;q)_\infty}{(xt;q)_\infty}
=\sum_{n=0}^\infty \frac{U_n^{(\alpha)}(x;q)}{(q;q)_n} t^n.
\]
These polynomials have 1 free parameter and 1 known generating function.
The power collection method will produce 1 connection relation for the 
free parameter, but no generalized generating functions 
since the above generating function is the only known generating 
function for Al-Salam-Carlitz I polynomials.
\item
Al-Salam-Carlitz II polynomials.
The relevant generating function is
\cite[(14.25.11)]{Koekoeketal}
\[
\frac{(xt;q)_\infty}{(t,at;q)_\infty}
=\sum_{n=0}^\infty \frac{(-1)^n q^{\binom{n}{2}}}{(q;q)_n} V_n^{(\alpha)}(x;q) t^n.
\]
These polynomials have 1 free parameter and 2 known generating functions.
The power collection method will produce 1 connection relation for the free parameter.
\end{itemize}

\section{Connection and connection-type relations}
\label{section3_Meixner}

The Meixner polynomials are defined as \cite[(9.10.1)]{Koekoeketal}
\begin{equation}
M_n(x;\alpha,c):=
\hyp21{-n,-x}{\alpha}{1-\frac{1}{c}}.
\label{3:7}
\end{equation}

In this section we derive and discuss connection and connection-type 
(see Remark \ref{rem:3}) relations for Meixner polynomials. 
For the entire paper, we assume that $x\in{\mathbb C}$, $n\in{\mathbb N}_0$.
Even though the power collection method may be used to derive the following
connection relations for Meixner polynomials, these can also be found 
(with proofs) in Gasper (1974) \cite[(5.2-5)]{Gasper74}.
\begin{thm}  \label{thm:1}
Let $\alpha, \beta\in {\widehat{\mathbb C}}, $ 
$c, d\in {{\mathbb C}_{0,1}}$. Then 
\begin{equation}
M_n(x;\alpha,c)
=\sum_{k=0}^n\binom{n}{k}
\frac{(\beta)_k}{(\alpha)_k}
\left(\frac{d(1-c)}{c(1-d)}\right)^k\,
\Hyper{2}{1}{-n+k,k+\beta}{k+\alpha}{\frac{d(1-c)}{c(1-d)}}
M_k(x;\beta,d).
\label{3:8}
\end{equation}
\end{thm}
\noindent By setting $\beta=\alpha$ in (\ref{3:8}) one obtains the following 
specialized result.
Let $\alpha\in {\widehat{\mathbb C}}$,  
$c, d\in {{\mathbb C}_{0,1}}$. Then 
\begin{equation} \label{3:9}
M_n(x;\alpha,c)= \left(\frac{c-d}{c(1-d)}\right)^n \sum_{k=0}^n \binom{n}{k}\left(
\frac{d(1-c)}{c-d}\right)^{k}M_k(x;\alpha,d).
\end{equation}
\noindent Furthermore by setting $d=c$ in (\ref{3:8}), and
using the Gauss formula \cite[(15.4.20)]{NIST}, one also has the following specialized result.
Let $\alpha, \beta\in {\widehat{\mathbb C}}$,  
$c\in {{\mathbb C}_{0,1}}$. Then 
\begin{equation} \label{3:10}
M_{n}(x;\alpha,c)
=\frac{1}{(\alpha)_n}\sum_{k=0}^{n}\binom{n}{k}(\alpha-\beta)_{n-k}(\beta)_k
M_k(x;\beta,c).
\end{equation}
\begin{remark} \label{rem:1}
Note that even though Theorem \ref{thm:1} was originally stated in 
\cite{Gasper74} for $\alpha>0$, $c\in(0,1)$, one can extend 
\cite[(5.9-12)]{Gasper74} analytically for $\alpha, \beta\in{\mathbb C}$, 
$-\alpha, -\beta\not \in \mathbb N_0$,  $c,d\in {{\mathbb C}_{0,1}}$, since, in such a 
case, one loses normality of the polynomials, i.e., $\deg M_n(x)<n$ for 
some $n$. 
However, formally,  Theorem \ref{thm:1} remains true for $c=1$, and all 
$\beta, d$ in the above domains.
\end{remark}

\begin{remark} \label{rem:3}
By {\it connection-type} relations for orthogonal polynomials, we mean a 
relation where  the left-hand side is an orthogonal polynomial with argument 
$x$ and set of parameters ${\bf a}$, and the right-hand side is given by a 
finite sum over coefficients which in general may depend on $x$, multiplied 
by a product of  that same polynomial with a set of different parameters 
${\bf b}$, namely
\[
P_n(x;{\bf a})=\sum_{k=0}^n \alpha_{k,n}(x;{\bf a},{\bf b}) P_k(x;{\bf b}).
\]
Connection-type relations are not connection relations (nor are they
unique) because the 
coefficients multiplying  the orthogonal polynomials depend on the argument.  
For connection relations, the coefficients  of the orthogonal polynomials must 
not depend on the argument.
\end{remark}
We now derive a connection-type relation for Meixner polynomials 
corresponding to the parameter $c$,
using the power collection method.
\begin{thm} \label{thm:4}
Let $\alpha \in {\widehat{\mathbb C}}$,  
$c, d\in {{\mathbb C}_{0,1}}$. Then 
\begin{equation}
\label{3:11}
M_{n}(x;\alpha,c)
=\frac{1}{(\alpha)_n}\sum_{k=0}^n \binom{n}{k}\frac{(\alpha)_k (x)_{n-k}}
{d^{n-k}} \, \hyp21{-n+k,-x}{-x+k-n+1}{\frac{d}{c}} M_k(x;\alpha,d).
\end{equation}
\end{thm}

\begin{proof} A generating function for Meixner polynomials is given
as \cite[(9.10.11)]{Koekoeketal}
\begin{equation}
\label{3:12}
\left(1-\frac{t}{c}\right)^{x}(1-t)^{-x-\alpha}=
\sum_{n=0}^\infty\frac{(\alpha)_n}{n!}M_n(x;\alpha,c)t^n,\qquad |t|<|c|<1.
\end{equation}
The above connection-type relation (\ref{3:11}) can be derived by starting 
with  (\ref{3:12}), and multiplying the left-hand side by
$\left(1-\dfrac{t}{d}\right)^{x} \Big/ \left(1-\dfrac{t}{d}\right)^{x},$ $|t|<|d|<1$.
Then, the left-hand side becomes
\begin{eqnarray}
&& \left(1-\dfrac{t}{c}\right)^{x}
\left(1-\dfrac{t}{d}\right)^{-x}
\left(1-\frac{t}{d}\right)^{x}(1-t)^{-x-\alpha}\nonumber\\
&&\hspace{4cm}=\sum_{m=0}^\infty\frac{(-x)_m}{m!}\left(\frac{t}{c}\right)^m
\sum_{s=0}^\infty\frac{(x)_s}{s!}\left(\frac{t}{d}\right)^s
\sum_{k=0}^\infty\frac{(\alpha)_k}{k!}M_k(x;\alpha,d)t^k,
\label{neweqn}
\end{eqnarray}
where the first two terms have been replaced using the binomial theorem
(\ref{2:6}), and the final two terms with the generating 
function (\ref{3:12}) with $c$ replaced by $d$.
Let $s=n-k-m$, and collect the terms associated with $t^n$ using
(\ref{3:12}) where the left-hand side has been re-expressed using 
(\ref{neweqn}). Then
(\ref{3:11}) follows using
analytic contination in $c$, $d$,
and (\ref{pochsum}), (\ref{negpoch}) and (\ref{2:5}).
\end{proof}

We now derive an interesting connection-type relation for Meixner polynomials 
corresponding to free parameters $\alpha$, $c$. The theorem below is 
not a connection relation because the coefficients multiplied by the Meixner 
polynomials depend on $x$  (see Remark \ref{rem:3}).
\begin{thm} \label{thm:5}
Let $\alpha, \beta\in {\widehat{\mathbb C}}$,  
$c, d\in {{\mathbb C}_{0,1}}$. Then 
\begin{equation} 
M_n(x;\alpha,c) 
= \frac{(\alpha-\beta)_n}
{(\alpha)_n}\sum_{k=0}^{n}\frac {(\beta)_k(-n)_k}{k!(\beta-\alpha-n+1)_k} 
F_1\!\left(\!-n+k,-x,x;\beta-\alpha-n+k+1;\frac{1}{c},\frac{1}{d}\right)
M_k(x;\beta,d),  \label{3:13}
\end{equation}
where $F_1$ is given by (\ref{3:14}).
\end{thm}

\begin{proof}
We substitute the connection relation for the free parameter $\alpha$
(\ref{3:10}) with the connection-type relation for the
free parameter $d$ (\ref{3:11}) to obtain the result
\[
M_n(x;\alpha,c)=\frac{1}{(\alpha)_n}\sum_{k=0}^{n}k!\binom{n}{k}(\alpha-\beta)_{n-k}
\sum_{m=0}^k\frac{(\beta)_m (x)_{k-m}}{m! (k-m)! d^{m-k}}
\hyp21{-k+m,-x}{-x+m-k+1}{\frac{d}{c}} M_m(x;\beta,d).
\]
If we expand the hypergeometric, switch the order of summations twice, and use 
(\ref{pochsum}), (\ref{negpoch}), (\ref{2:5}), (\ref{3:14}), the result follows.
\end{proof}

Krawtchouk polynomials are a particular case of 
Meixner polynomials. In fact, they are related in the following way
\begin{equation} \label{relmex-krw}
K_n(x;p,N)=M_n\left(x;-N,\frac{p}{p-1}\right).
\end{equation}
Taking this into account, we can write them as a truncated 
hypergeometric series as \cite[(9.11.1)]{Koekoeketal}
\begin{equation}
K_n(x;p,N):=
\hyp21{-n,-x}{-N}{\frac{1}{p}}.
\label{Kbhs}
\end{equation}
The following connection results for Krawtchouk polynomials can be 
found in \cite[(5.9-10), (5.11-12)]{Gasper74}.
\begin{thm} \label{thm:6}
Let $n, M, N\in{\mathbb N}_0$, $n\le N\le M$, $p, q\in {{\mathbb C}_{0}}$. Then
\begin{equation}
K_n(x;p,N)
=\sum_{k=0}^n\binom{n}{k}
\frac{ q^k (-M)_k}{p^k(-N)_k}
\,\hyp21{-n+k,k-M}{k-N}{\frac{q}{p}}
K_k(x;q,M).
\label{KrawCNXN2}
\end{equation}
\end{thm}
Setting $M=N$ in (\ref{KrawCNXN2}) one obtains the following connection result.
Let $n, N\in{\mathbb N}_0$, $n\le N$, $p, q\in {{\mathbb C}_{0}}$. Then
\begin{equation} \label{KrawCNXN1}
K_n(x;p,N)= \left(\frac{p-q}{p}\right)^n \sum_{k=0}^n \binom{n}{k}\left(
\frac{q}{p-q}\right)^k K_k(x;q,N).
\end{equation}
Furthermore by setting $d=c$ in (\ref{KrawCNXN2}) and
using the Gauss formula \cite[(15.4.20)]{NIST}, one obtains the following.
Let $n, M, N\in{\mathbb N}_0$, $n\le N\le M$, $p, q\in {{\mathbb C}_{0}}$. Then
\begin{equation}
\label{KrawCNXN3}
K_{n}(x;p,N)
=\frac{1}{(-N)_n}\sum_{k=0}^{n}\binom{n}{k}(M-N)_{n-k}(-M)_k
K_k(x;p,M).
\end{equation}
\begin{remark} \label{rem:2}
Observe that the results for Krawtchouk polynomials presented in this paper
may also be 
obtained by starting with (\ref{3:7}), setting the right values, 
and using the relation (\ref{relmex-krw}).
\end{remark}

\section{Generalized generating functions from connection(-type) relations}
\label{GeneralizedGeneratingFunctions}
We now combine generating functions for Meixner and Krawtchouk polynomials with the
above connection and connection-type relations to derive generalized generating functions.
First we derive generalized generating functions for the Meixner polynomials.
\begin{thm} \label{thm:9}
Let $\alpha, \beta\in{\widehat{\mathbb C}}$, $c, d\in {{\mathbb C}_{0,1}}$, $x,t\in{\mathbb C}$. Then 
\begin{equation}
\hyp11{-x}{\alpha}{\frac{t(1-c)}{c}}
=\sum_{n=0}^{\infty}\frac{(\beta)_n}{(\alpha)_n n!}\left(\frac{d(1-c)}{c(1-d)}
\right)^n \hyp11{\beta+n}{\alpha+n}{\frac{-td(1-c)}{c(1-d)}}M_n(x;\beta,d) t^n.
\label{MeixgenfunG2FP}
\end{equation}
\end{thm}

\begin{proof}
Using the generating function for Meixner polynomials \cite[(9.10.12)]{Koekoeketal}
\begin{equation}
e^t\hyp11{-x}{\alpha}{\frac{t(1-c)}{c}}=
\sum_{n=0}^{\infty}M_n(x;\alpha,c)
\frac{t^n}{n!} 
\label{KLS91012}
\end{equation}
and (\ref{3:8}), we obtain
\begin{eqnarray*}
&& \hspace{-0.5cm}e^t\hyp11{-x}{\alpha}{\frac{t(1-c)}{c}}\nonumber \\
&& \hspace{0.5cm}=\sum_{n=0}^{\infty}\frac{t^n}{n!}\sum_{k=0}^{n}\binom n k\frac{(\beta)_k}{(\alpha)_k}
\left(\frac{d(1-c)}{c(1-d)}\right)^k
\hyp21{-n+k,\beta+k}{\alpha+k}{\frac{d(1-c)}{c(1-d)}}  M_k(x;\beta,d).
\end{eqnarray*}
\noindent If we switch the order of summations, shift the $n$ variable by a factor of $k$, expand 
the hypergeometric, switch the order of summations again, and use 
(\ref{pochsum}), (\ref{negpoch}) and (\ref{2:5}). Again, in order to justify reversing the 
summation symbols it is enough to show that
\[
\sum_{n=0}^\infty |a_n|\left|\sum_{k=0}^n c_{k,n} M_k(x;\beta,d)\right|<\infty,
\]
where $|M_k(x,\beta,d)|\le K_1(1+k)^{\sigma_2}d^{-k}$,
$a_n=t^n/n!,$ hence $|a_n|\le |t|^n/n!,$
and
\[
c_{k,n}=\sum_{s=0}^{n-k} \frac{(-1)^k(-n)_{s+k}(\beta)_{s+k}}{(\alpha)_{s+k}k!s!} 
\left(\frac{d(1-c)}{c(1-d)}\right)^{s+k},
\]
where $K_1$ and $\sigma_1$ are positive constants not depending on $n$.
Then since 
\[
\sum_{n=0}^\infty |a_n|\left|\sum_{k=0}^n c_{k,n} M_k(x;\beta,d)\right|\le 
K_1K_2 \sum_{n=0}^\infty \frac{(1+n)^{\sigma_1+\sigma_2+1}}{n!} \left|\frac{t}{c}\right|^n
\left|\frac{1+d-2c}{1-d}\right|^n<\infty,
\]
the result follows because all the sums connected with these 
coefficients converge.
\end{proof}
A direct consequence 
of Theorem \ref{thm:9} with $c=d$, 
and \cite[(13.2.39)]{NIST} is
given as follows.
Let $\alpha, \beta\in {\widehat{\mathbb C}}$, $c\in {{\mathbb C}_{0,1}} $, $x,t\in{\mathbb C}$. Then 
\begin{equation}
e^t\hyp11{-x}{\alpha}{\frac{t(1-c)}{c}}
=\sum_{n=0}^{\infty}\frac{(\beta)_n}{(\alpha)_n n!}
\,\hyp11{\alpha-\beta}{\alpha+n}{t}M_n(x;\beta,c)t^n.
\label{Meixgenfun1}
\end{equation}


We now
combine Meixner generating function 
(\ref{KLS91012})
with the
connection-type relation 
(\ref{3:13})
to derive a generalized generating function.
\begin{thm} \label{thm:11}
Let $\alpha\in {\widehat{\mathbb C}}$, $c, d\in {{\mathbb C}_{0,1}}$, $x,t\in{\mathbb C}$. Then 
\begin{equation}
\label{Meixgenfun3}
e^t \hyp11{-x}{\alpha}{\frac{t(1-c)}{c}}
=\sum_{n=0}^{\infty}\frac{1}{n!}\Phi_2\left(x,-x;\alpha+n;\frac{t}{c},
\frac{t}{d}\right)M_n(x;\alpha,d)t^n,
\end{equation}
where $\Phi_2$ is given by (\ref{kummer}).
\end{thm}

\begin{proof}
Using 
(\ref{KLS91012})
and (\ref{3:11}), we obtain
\begin{equation} \label{4:29}
e^t \hyp11{-x}{\alpha}{\frac{t(1-c)}{c}} =\sum_{n=0}^{\infty}\frac{t^n}{(\alpha)_n}
\sum_{k=0}^n\frac{(\alpha)_k (x)_{n-k}}{k! (n-k)! d^{n-k}} M_k(x;\alpha,d) 
\,\hyp21{-n+k,-x}{-x+k-n+1}{\frac{d}{c}}.
\end{equation}

\noindent Switch the order of the summations based on $n$ and $k$, shift the $n$ 
variable by a factor of $k$, expand the hypergeometric, and use 
(\ref{pochsum}), (\ref{negpoch}), (\ref{2:5}), and (\ref{kummer}). 
We can justify the reversing the summation symbols since in this case
\[
a_n=\frac {t^n}{n!},\quad \text{and} \quad
c_{k,n}=\binom n k \frac{(\alpha)_k (x)_{n-k}}
{d^{n-k}} \,\hyp21{-n+k,-x}{-x+k-n+1}{\frac{d}{c}}.
\]
Therefore
\[
\sum_{n=0}^\infty |a_n|\left|\sum_{k=0}^n c_{k,n} M_k(x;\alpha,d)\right|\le 
K_3 \sum_{n=0}^\infty \frac{(1+n)^{\sigma_3}}{n!} \left|\frac{t}{c}\right|^n,
\]
where $K_3$ and $\sigma_3$ are positive constants not depending on $n$, then 
the result holds since all these sums connected with these coefficients converge.
\end{proof}
\begin{thm} \label{thm:12}
Let $\alpha, \beta\in {\widehat{\mathbb C}}$, $c, d\in {{\mathbb C}_{0,1}}$, $x,t\in{\mathbb C}$. Then 
\begin{equation}
\label{Meixgenfun3_1}
e^t
\hyp11{-x}{\alpha}{\frac{t(1-c)}{c}}
=\sum_{n=0}^{\infty} \frac{(\beta)_n}{(\alpha)_n n!} 
\Phi_2^{(3)}\left(x,-x,\alpha-\beta;\alpha+n;\frac{t}{c},\frac{t}{d},t\right)
M_n(x;\beta,d) t^n,
\end{equation}
where $\Phi_2^{(3)}$ is given in (\ref{kummer_3}).
\end{thm}

\begin{proof}
Using (\ref{KLS91012})
and (\ref{3:13}), we obtain
\begin{equation} \label{4:32}
\begin{split}
e^t
\hyp11{-x}{\alpha}{\frac{t(1-c)}{c}}
=\sum_{n=0}^{\infty}\frac{t^n}{n!}
\frac{(\alpha-\beta)_n}{(\alpha)_n}\sum_{k=0}^n
\frac{(\beta)_k(-n)_k}{k!(\beta-\alpha-n+1)_k}M_k(x;\beta,d)\\
\times  F_1\left(-n+k,-x,x;\beta-\alpha-n+k+1;\frac{1}{c},\frac{1}{d}\right).
\end{split}
\end{equation}
Switch the order of the summations based on $n$ and $k$, shift the $n$ variable 
by a factor of $k$, expand the Appell series, switch the order of summations 
two more times, and use (\ref{pochsum}), (\ref{negpoch}), (\ref{2:5}), and 
(\ref{kummer_3}). Indeed, 
\[
\sum_{n=0}^\infty |a_n|\left|\sum_{k=0}^n c_{k,n} M_k(x;\beta,d)\right|\le 
K_4 \sum_{n=0}^\infty \frac{(1+n)^{\sigma_4}}{n!} \left|\frac{t(c+d)}{cd}\right|^n<\infty,
\]
where $K_4$ and $\sigma_4$ are positive constants not depending on $n$, 
then  the result holds since all these sums connected with these coefficients 
can be rearranged in the desired way.
\end{proof}

\noindent We also have the connection relation with one free parameter given by
(\ref{3:10}). We now combine this connection relation with the above referenced
generating functions to obtain new generalized
generating functions for Meixner polynomials.

\begin{thm} \label{thm:13}
Let $c\in {{\mathbb C}_{0,1}}$, $\gamma, t\in \mathbb C$, $|t|<1$, $|t(1-c)|<|c(1-t)|$, $\alpha, \beta\in {\widehat{\mathbb C}}$. Then 
\begin{equation}
(1-t)^{-\gamma}
\hyp21{\gamma,-x}{\alpha}{\frac{t(1-c)}{c(1-t)}}
=\sum_{n=0}^{\infty}\frac{(\gamma)_n (\beta)_n}{(\alpha)_n n!}
\,\hyp21{\gamma+n,\alpha-\beta}{\alpha+n}{t}M_n(x;\beta,c)t^n.
\label{Meixgenfun2}
\end{equation}
\end{thm}

\begin{proof}
Using the generating function for Meixner polynomials 
\cite[(9.10.13)]{Koekoeketal} and (\ref{3:10}), we obtain
\[
(1-t)^{-\gamma}
\hyp21{\gamma,-x}{\alpha}{\frac{t(1-c)}{c(1-t)}}
=\sum_{n=0}^{\infty}\frac{(\gamma)_n t^n}{(\alpha)_{n}}\sum_{k=0}^n
\frac{(\alpha-\beta)_{n-k}(\beta)_{k}}{(n-k)!k!}M_k(x;\beta,c).
\]
\noindent If we switch the order of summations, shift the $n$ variable by 
a factor of $k$ and use (\ref{pochsum}), (\ref{negpoch}) and (\ref{2:5}). 
Indeed, in this case $a_n=t^n (\gamma)_n/n!$, therefore
\[
|a_n|\le |t|^n (1+n)^{|\gamma|}.
\]
So, we have
\[
\sum_{n=0}^\infty |a_n|\left|\sum_{k=0}^n c_{k,n} M_k(x;\beta,c)\right|\le 
K_5 \sum_{n=0}^\infty (1+n)^{\sigma_5} \left|\frac{t(1-c)}{c(1-t)}\right|^n,
\]
where $K_5$ and $\sigma_5$ are positive constants not depending on $n$. 
Therefore if $|t|<1$ and $|t(1-c)|<|c(1-t)|$ the sum converges, then 
the result holds since all these sums connected with these coefficients 
can be rearranged in the desired way.
\end{proof}
\begin{thm}  \label{thm:14}
Let $c,d \in {{\mathbb C}_{0,1}}$, $\gamma, t\in \mathbb C$, $|t|<\min\{1, |c(1-d)|/|1+d-2c|\}$, 
$\alpha, \beta\in {\widehat{\mathbb C}}$.  Then 
\begin{equation}
\begin{split}
(1-t)^{-\gamma}
\hyp21{\gamma,-x}{\alpha}{\frac{t(1-c)}{c(1-t)}}
=\sum_{n=0}^{\infty}\frac{(\gamma)_n(\beta)_n}{(\alpha)_n n!}
\,\hyp21{\gamma+n,\beta+n}{\alpha+n}{\frac{-dt(1-c)}{c(1-d)(1-t)}}
\\ \times
\left(\frac{d(1-c)}{c(1-d)(1-t)}\right)^n M_n(x;\beta,d)t^n.
\label{MeixgenfunG3FP}
\end{split}
\end{equation}
\end{thm}

\begin{proof}
Using \cite[(9.10.13)]{Koekoeketal} and (\ref{3:8}), we obtain
\begin{eqnarray*}
&& (1-t)^{-\gamma}
\hyp21{\gamma,-x}{\alpha}{\frac{t(1-c)}{c(1-t)}}
=\sum_{n=0}^{\infty}\frac{(\gamma)_n}{n!}t^n\sum_{k=0}^{n}
\frac{ (\beta)_k n!}{k! (n-k)!(\alpha)_k} \left(\frac{d(1-c)}{c(1-d)}\right)^k 
M_k(x;\beta,d) \\[0.2cm]
&& \hspace{10cm}\times \hyp21{-n+k,\beta+k}{\alpha+k}{\frac{d(1-c)}{c(1-d)}}.
\end{eqnarray*}
\noindent If we switch the order of summations, shift the $n$ variable by 
a factor of $k$, expand the hypergeometric, switch the order of summations 
again, and use (\ref{pochsum}), (\ref{negpoch}) and (\ref{2:5}), then the result 
holds since all these sums connected with  these coefficients converge (it is similar 
to the previous proof combined with the proof of Theorem \ref{thm:9}) and can be 
rearranged in the desired way.
\end{proof}

%

\noindent Above, we have found a finite expansion of the Meixner 
polynomials with free parameter $c$ in terms of Meixner polynomials 
with free parameter $d$ (see the connection-type relation (\ref{3:11})). 
We now combine Meixner generating function 
\cite[(9.10.13)]{Koekoeketal} with that connection-type relation to derive 
a generalized generating function whose coefficient is an Appell $F_1$ double
hypergeometric function.
\begin{thm} \label{thm:15}
Let $|t|<\min\{1,|c|\}, \alpha\in{\widehat{\mathbb C}}$, $\gamma\in\mathbb C$, $c, d\in{{\mathbb C}_{0,1}}$. Then
\begin{equation}
\begin{split}
\label{Meixgenfun4}
(1-t)^{-\gamma} \hyp21{\gamma,-x}{\alpha}{\frac{t(1-c)}{c(1-t)}}=\sum_{n=0}^{\infty}
\frac{(\gamma)_n }{n!}  F_1\left(\gamma+n,x,-x;\alpha+n;\frac{t}{c},\frac{t}{d}\right)
M_n(x;\alpha,d)t^n.
\end{split}
\end{equation}
\end{thm}

\begin{proof}
Using \cite[(9.10.13)]{Koekoeketal} and (\ref{3:11}), we obtain

\begin{equation*}
\begin{split}
(1-t)^{-\gamma} \hyp21{\gamma,-x}{\alpha}{\frac{t(1-c)}{c(1-t)}}
=\sum_{n=0}^{\infty}\frac{(\gamma)_n t^n}{(\alpha)_n n!}\sum_{k=0}^n
\frac{(\alpha)_k (x)_{n-k}}{k! (n-k)! d^{n-k}} M_k(x;\alpha,d) \\ \times
\hyp21{-n+k,-x}{-x+k-n+1}{\frac{d}{c}}.
\end{split}
\end{equation*}
\noindent Switch the order of the summations based on $n$ and $k$, shift the 
$n$ variable by a factor of $k$, expand the hypergeometric, and use 
(\ref{pochsum}), (\ref{negpoch}), (\ref{2:5}), and (\ref{3:14}), then the result 
holds since all these sums connected with  these coefficients converge (it is similar 
to the proof of Theorem \ref{thm:13} combined with the proof of Theorem \ref{thm:11}) 
and can be rearranged in the desired way.
\end{proof}
\begin{thm} \label{thm:16}
Let $|t|<\min\{1, |cd|/|c+d|\}, \alpha, \beta\in{\widehat{\mathbb C}}$, $\gamma\in\mathbb C$, 
$c, d\in{{\mathbb C}_{0,1}}$. Then
\begin{equation}
\begin{split}
\label{Meixgenfun3_2}
(1-t)^{-\gamma} \hyp21{\gamma,-x}{\alpha}{\frac{t(1-c)}{c(1-t)}}
=\sum_{n=0}^{\infty} \frac{(\beta)_n (\gamma)_n}{(\alpha)_n n!} F_D^{(3)}
\left(\gamma+n,x,-x,\alpha-\beta;\alpha+n;\frac{t}{c},\frac{t}{d},t\right)
\\ \times M_n(x;\beta,d) t^n,
\end{split}
\end{equation}
where $F_D^{(3)}$ is given in (\ref{Appell_D}).
\end{thm}

\begin{proof}
Using \cite[(9.10.13)]{Koekoeketal} and (\ref{3:13}), we obtain

\begin{equation}
\begin{split}
(1-t)^{-\gamma} \hyp21{\gamma,-x}{\alpha}{\frac{t(1-c)}{c(1-t)}}
=\sum_{n=0}^{\infty}\frac{(\gamma)_n t^n}{n!}
\frac{(\alpha-\beta)_n}{(\alpha)_n}\sum_{k=0}^{\infty}\frac{(\beta)_k(-n)_k}
{k!(\beta-\alpha-n+1)_k} M_k(x;\beta,d)\\
\times  F_1\left(-n+k,-x,x;\beta-\alpha-n+k+1;\frac{1}{c},\frac{1}{d}\right).
\end{split}
\end{equation}
Switch the order of the summations based on $n$ and $k$, shift the $n$ variable 
by a factor of $k$,expand the Appell series, switch the order of summations 
two more times, and use (\ref{pochsum}), (\ref{negpoch}), (\ref{2:5}), and 
(\ref{Appell_D}), then the result  holds since all these sums connected with 
these coefficients converge (it is similar to the proof of Theorem \ref{thm:13} 
combined with the proof of Theorem \ref{thm:18}) 
and can be rearranged in the desired way.
\end{proof}

\noindent We have derived generalized generating functions for the free 
parameter $c$. However, since the coefficients of our connection-type 
relation is in terms of $x$, we cannot use the orthogonality
relation to create new infinite sums.
Note that the application of connection relations (\ref{3:10}) and 
(\ref{3:11}) to the rest of the known generating functions for Meixner polynomials
\cite[(9.10.11-13)]{Koekoeketal}
leave these generating functions invariant.\\[-0.2cm]

We now derive generalized generating functions for the Krawtchouk polynomials,
where we will need a special notation for some
of the generating functions. Let $f\in C^\infty(\mathbb C)$, $N\in{\mathbb N}_0$, $t\in{\mathbb C}$. Define 
the truncated Maclaurin expansion of $f$ as
(cf.~\cite[p.~6]{Koekoeketal})
\[
\left[f(t)\right]_N:=\sum_{k=0}^N \frac{f^{(k)}(0)}{k!} t^k.
\]
\begin{thm} \label{thm:18}
Let $p, q \in {{\mathbb C}_{0}}$, $M,N\in \mathbb N_0$, $N\le M$, $x,t\in{\mathbb C}$. Then
\begin{equation}
\label{KrawGenFun1_2}
\left[e^{t}\hyp11{-x}{-N}{-\frac{t}{p}}\right]_N
=\sum_{n=0}^{N}
\frac{(-M)_n}{(-N)_n n!}
\left(\frac{tq}{p}\right)^n
\left[e^t\hyp11{n-M}{n-N}{\frac{-tq}{p}}\right]_{N-n}K_n(x;q,M).
\end{equation}
\end{thm}

\begin{proof}
Using 
\cite[(9.11.12)]{Koekoeketal} 
\begin{equation}
\left[e^{t}\hyp11{-x}{-N}{-\frac{t}{p}}\right]_N
=\sum_{n=0}^N \frac{t^n}{n!}K_n(x;p,N),
\label{KLS91112}
\end{equation}
and (\ref{KrawCNXN1}), we obtain
\begin{equation}
\left[e^{t}\hyp11{-x}{-N}{-\frac{t}{p}}\right]_N
=\sum_{n=0}^N \frac{t^n}{n!}\sum_{k=0}^n \binom{n}{k}
\frac{(-M)_k
}
{(-N)_k 
}
\left(\frac{q}{p}\right)^k
\,\hyp21{-n+k,k-M}{k-N}{\frac{q}{p}}
K_k(x;q,M).
\end{equation}
If we switch the order of summations, shift the $n$ variable by a factor of $k$, 
expand the hypergeometric, then switch the order of summations again and shift 
the $n$ variable again, and use  (\ref{pochsum}), (\ref{negpoch}) and (\ref{2:5}), 
the proof follows since all the series have finite number of terms.
\end{proof}
Letting $p=q$ in (\ref{KrawGenFun1_2})
yields the following result.
Let $p \in {{\mathbb C}_{0}}$, $M,N\in \mathbb N_0$, $N\le M$, $x,t\in{\mathbb C}$. Then\begin{equation}
\label{KrawGenFun1_1}
\left[e^{t}\hyp11{-x}{-N}{-\frac{t}{p}}\right]_N=\sum_{n=0}^{N}
\frac{
(-M)_n
t^n
}{
(-N)_n
n!
}
\,\left[\hyp11{M-N}{n-N}{t}
\right]_{N-n}
K_n(x;p,M)
.
\end{equation}
Furthermore, letting $M=N$ in (\ref{KrawGenFun1_2}) produces 
the following.
Let $p,q \in {{\mathbb C}_{0}}$, $N\in \mathbb N_0$, $x,t\in{\mathbb C}$. Then\begin{equation}
\left[e^{t}\hyp11{-x}{-N}{-\frac{t}{p}}\right]_N=
\sum_{n=0}^{N}
\frac{1}{n!}\left(\frac{tq}{p}\right)^n
\left[e^{t(1-q/p)}\right]_{N-n}
K_n(x;q,N).
\end{equation}

\begin{thm} \label{thm:20}
Let $p, q \in {{\mathbb C}_{0}}$, $M,N\in \mathbb N_0$, $N\le M$, $t, \gamma \in{\mathbb C}$. Then
\begin{eqnarray}
\label{KrawGenFun2_2}
&&\hspace{-1.5cm}\left[(1-t)^{-\gamma}\hyp21{\gamma,-x}{-N}{\frac{t}{p(t-1)}}\right]_N
=\sum_{n=0}^{N}
\left(\frac{tq}{p}\right)^n \frac{(-M)_n (\gamma)_n}{(-N)_n n!}
\nonumber\\[0.2cm]
&&\hspace{4.5cm}\times\left[(1-t)^{-\gamma-n}
\hyp21{\gamma+n,n-M}{n-N}{\frac{-qt}{p(1-t)}}
\right]_{N-n}
K_n(x;q,M).
\end{eqnarray}
\end{thm}

\begin{proof}
Using 
\cite[(9.11.13)]{Koekoeketal}
\begin{equation}
\left[(1-t)^{-\gamma}\hyp21{\gamma,-x}{-N}{\frac{t}{p(t-1)}}
\right]_N
=\sum_{n=0}^N
\frac{(\gamma)_nt^n}{n!}K_n(x;p,N),
\label{KLS91113}
\end{equation}
where $\gamma\in{\mathbb C}$, 
and (\ref{KrawCNXN1}), we obtain
\begin{equation}
\begin{split}
\left[(1-t)^{-\gamma}\hyp21{\gamma,-x}{-N}{\frac{t}{p(t-1)}}
\right]_N
=\sum_{n=0}^N \frac{(\gamma)_n}{n!} t^n\sum_{k=0}^n \binom{n}{k}\frac{(-M)_k q^k}{(-N)_k p^k}
\,\hyp21{-n+k,k-M}{k-N}{\frac{q}{p}}
\\ \times K_k(x;q,M).
\end{split}
\end{equation}
If we switch the order of summations, shift the $n$ variable by a factor of 
$k$, expand the hypergeometric, then switch the order of summations again 
and shift the $n$ variable again, and use (\ref{pochsum}), (\ref{negpoch}) 
and (\ref{2:5}), the proof follows since all the series have finite number of terms.
\end{proof}

\noindent If we let $p=q$ in (\ref{KrawGenFun2_2}) and use
\cite[(15.8.1)]{NIST}, we obtain the following
result.
Let $p \in {{\mathbb C}_{0}}$, $M, N\in \mathbb N_0$, $N\le M$, 
$x,t, \gamma \in{\mathbb C}$. Then
\[
\left[(1-t)^{-\gamma}\hyp21{\gamma,-x}{-N}{\frac{t}{p(t-1)}}\right]_N
=\sum_{n=0}^{N}\frac{(-M)_n (\gamma)_nt^n}{(-N)_n n!}
\left[\hyp21{\gamma+n,M-N}
{n-N}{t}\right]_{N-n}
K_n(x;p,M).
\]
If we let $M=N$ in (\ref{KrawGenFun2_2}) we obtain the following.
Let $p,q \in {{\mathbb C}_{0}}$, $N\in \mathbb N_0$, 
$x,t, \gamma \in{\mathbb C}$. Then
\[
\left[(1-t)^{-\gamma}\hyp21{\gamma,-x}{-N}{\frac{t}{p(t-1)}}\right]_N
=\sum_{n=0}^{N}\frac{(\gamma)_n}{n!}\left(\frac{qt}{p}\right)^n
\left[\left(1+t\left(\frac{q}{p}-1\right)\right)^{-\gamma-n}
\right]_{N-n}
K_n(x;q,N).
\]

\noindent Note that the application of connection relations (\ref{KrawCNXN1}) 
and (\ref{KrawCNXN3}) to the generating
functions for Krawtchouk polynomials
\cite[(9.11.11-13)]{Koekoeketal}
leave these generating functions invariant.
\section{Results using orthogonality}
\label{ResultsUsingOrthogonality}

\noindent We have derived generalized generating functions for the free parameter
$\alpha$. We now combine this with the orthogonality relation for Meixner polynomials
to produce new results from our generalized generating functions.
The well-known orthogonality relation for Meixner polynomials 
for $n,m\in{\mathbb N}_0$, $\alpha>0,$ $c\in(0,1)$ is
\cite[(14.25.2)]{Koekoeketal}
\begin{equation}
\label{5:51}
\sum_{x=0}^\infty
M_n(x;\alpha,c)M_m(x;\alpha,c)\frac{\Gamma(x+\alpha)c^x}{\Gamma(x+1)}=
\kappa_n\delta_{m,n},
\end{equation}
where
\[
\kappa_n= \frac{n!}{c^n(1-c)^{\alpha}(\alpha)_n}.
\]
Note that this a particular case of a more general property of orthogonality
fulfilled by Meixner polynomials (see \cite[Proposition 9]{cola1}). 
\begin{prop} \label{prop:21} 
Let $m,n\in \mathbb N_0$, 
$\alpha\in {\widehat{\mathbb C}} $,
$c\in \mathbb C\setminus[0,\infty)$.  The orthogonality relation for Meixner 
polynomials can be given as
\begin{equation}
\label{5:52}
\int_C M_n(z;\alpha,c) M_m(x;\alpha,c) 
w(z;\alpha,c)
dz=
\kappa_n \delta_{m,n},
\end{equation}
where
\[
w(z;\alpha,c):=\Gamma(-z)\Gamma(z+\alpha)(-c)^z,
\]
and $C$ is a complex contour from $-\infty \mathit{i}$ to $\infty\mathit{i}$ 
separating the increasing poles at $z\in\mathbb N_0$ from the decreasing poles 
at $z\in\{-\alpha, -\alpha-1, -\alpha-2, \dots\}$.
\end{prop}

In fact, observe that the case $c>0$ cannot be considered by an integral 
of the form (\ref{5:52}) since   it diverges.  However, when $|c|<1$, 
(\ref{5:52}) is rewritten on the form (see \cite[Section 5.6]{temmebook1} 
for details) presented in (\ref{5:51}).
With this result in mind, the following result and corresponding consequences hold.
\begin{thm} \label{thm:22}
Let $t\in{\mathbb C}$, $\alpha, \beta \in {\widehat{\mathbb C}}$,
$c\in \mathbb C\setminus[0,\infty)$. Then
\begin{equation}
\label{Meixint}
\int_C \hyp11{-z}{\alpha}{\frac{t(1-c)}{c}}M_n(z;\beta,c)\Gamma(-z)
\Gamma(z+\alpha)(-c)^z\,dz=\frac{t^n\, e^{-t}}{(1-c)^{\beta} (\alpha)_n 
c^n}\hyp11{\alpha-\beta}{\alpha+n}{t}.
\end{equation}
\end{thm}
\begin{proof}
From (\ref{Meixgenfun1}) we multiply both sides by $M_m(z;\beta,c)w(z;\beta,c)$,
utilizing 
the orthogonality relation (\ref{5:52}), produces the desired result.
\end{proof}
\begin{cor} \label{cor:23}
Let $t\in{\mathbb C}$, $\alpha, \beta>0$, $c\in (0,1)$. Then
\begin{equation}
\label{Meixsum}
\sum_{x=0}^{\infty} \hyp11{-x}{\alpha}{\frac{t(1-c)}{c}}M_n(x;\beta,c)
\frac{(\beta)_x c^x}{x!}=\frac{t^n\, e^{-t}}{(1-c)^{\beta} (\alpha)_n c^n}
\,\hyp11{\alpha-\beta}{\alpha+n}{t}.
\end{equation}
\end{cor}
\begin{cor} \label{thm:24}
Let $t\in{\mathbb C}$, $\alpha, \beta \in {\widehat{\mathbb C}}$, $c, d\in {\mathbb C}\setminus[0,\infty)$. Then
\begin{equation}
\begin{split}
\label{Meixsum2FP_1}
\int_C \hyp11{-z}{\alpha}{\frac{t(1-c)}{c}}M_n(z;\beta,d)\Gamma(-z)
\Gamma(z+\beta)(-d)^z\,dz=\frac{t^n(1-c)^n e^{-t}}{(1-d)^{n+\beta} (\alpha)_n c^n}
\\ \times \hyp11{\beta+n}{\alpha+n}{\frac{-dt(1-c)}{c(1-d)}}.
\end{split}
\end{equation}
\end{cor}
\begin{proof}
From (\ref{MeixgenfunG2FP}) we multiply both sides by $M_m(z;\beta,c)
w(z;\beta,c)$, utilizing the orthogonality relation (\ref{5:52}).
\end{proof}
\begin{cor} \label{cor:25}
Let $t\in{\mathbb C}$, $\alpha, \beta>0$, $c, d\in (0,1)$. Then
\begin{equation}
\label{Meixsum2FP_1}
\sum_{x=0}^{\infty} \hyp11{-x}{\alpha}{\frac{t(1-c)}{c}}M_n(x;\beta,d)
\frac{d^x (\beta)_x}{x!}=\frac{t^n(1-c)^n e^{-t}}{c^n(1-d)^{n+\beta} (\alpha)_n}
\,\hyp11{\beta+n}{\alpha+n}{\frac{-dt(1-c)}{c(1-d)}}.
\end{equation}
\end{cor}
\begin{cor}\label{thm:26}
Let $c\in {\mathbb C}\setminus[0,\infty)$, $t\in{\mathbb C}$, $|t|<1$, $|t(1-c)|<|c(1-t)|$, $\alpha, 
\beta \in {\widehat{\mathbb C}}$, $\gamma\in \mathbb C$. Then
\begin{equation}
\label{Meixint-1}
\int_C \hyp21{\gamma,-z}{\alpha}{\frac{t(1-c)}{c(1-t)}}M_n(z;\beta,c)
\Gamma(z+\beta)(-c)^z\,dz=\frac{(1-t)^\gamma(\gamma)_nt^n}{(1-c)^{\beta}
(\alpha)_nc^n}\,\hyp21{\alpha-\beta,\gamma+n}{\alpha+n}{t}.
\end{equation}
\end{cor}
\begin{proof}
From (\ref{Meixgenfun2}) we multiply both sides by $M_m(z;\beta,c)
w(z;\beta,c)$, utilizing the orthogonality relation (\ref{5:52}).
\end{proof}
\begin{cor}\label{cor:27}
Let $c\in (0,1)$, $t\in{\mathbb C}$, $|t|<1$, $|t(1-c)|<|c(1-t)|$, $\alpha, \beta>0$, 
$\gamma\in \mathbb C$. Then
\begin{equation}
\label{Meixsum-1}
\sum_{x=0}^{\infty}
\hyp21{\gamma,-x}{\alpha}{\frac{t(1-c)}{c(1-t)}}M_n(x;\beta,c)\frac{c^x(\beta)_x}
{x!}=\frac{(1-t)^\gamma(\gamma)_n t^n}{(1-c)^{\beta}(\alpha)_nc^n} 
\,\hyp21{\alpha-\beta,\gamma+n}{\alpha+n}{t}.
\end{equation}
\end{cor}
\begin{cor} \label{thm:30}
Let $t\in{\mathbb C}$, $|t|<\min\{1, |c(1-d)|/|1+d-2c|\}, \alpha, \beta\in{\widehat{\mathbb C}}$,
$\gamma\in \mathbb C$, $c, d\in{\mathbb C}\setminus[0,\infty)$. Then
\begin{equation}
\begin{split}
\label{Meixint2FP_2}
\int_C \hyp11{\gamma,-z}{\alpha}{\frac{t(1-c)}{c(1-t)}}M_n(z;\beta,d)
\Gamma(z+\beta)(-d)^z\,dz=\frac{(\gamma)_n}{(1-d)^{n+\beta} (\alpha)_n}
\left(\frac{t(1-c)}{c(1-t)}\right)^n\\ \times\hyp21{\gamma+n,\beta+n}
{\alpha+n}{\frac{-dt(1-c)}{c(1-d)(1-t)}}.
\end{split}
\end{equation}
\end{cor}

\begin{proof}
From (\ref{MeixgenfunG3FP}) we multiply both sides by $M_m(z;\beta,c)w(z;\beta,c)$,
utilizing the orthogonality relation (\ref{5:51}), produces the desired result.
\end{proof}
\begin{cor} \label{cor:31}
Let $c, d\in (0,1)$, $t\in{\mathbb C}$, $|t|<\min\{1, |cd|/|c+d|\}, \alpha, \beta>0$, 
$\gamma\in \mathbb C$. Then
\begin{equation}
\begin{split}
\label{Meixsum2FP_2}
\sum_{x=0}^{\infty} \hyp11{\gamma,-x}{\alpha}{\frac{t(1-c)}{c(1-t)}}
M_n(x;\beta,d)\frac{d^x (\beta)_x}{x!}=\frac{(\gamma)_n}{(1-d)^{n+\beta} (\alpha)_n}
\left(\frac{t(1-c)}{c(1-t)}\right)^n\\ \times
\hyp21{\gamma+n,\beta+n}{\alpha+n}{\frac{-dt(1-c)}{c(1-d)(1-t)}}.
\end{split}
\end{equation}
\end{cor}
On the other hand, since the Krawtchouk polynomials satisfy the property of orthogonality
\[
\sum_{x=0}^N \binom N x p^x(1-p)^{N-x} K_m(x;p,N) K_n(x;p,N) 
=\frac{(-1)^n\,n!}{(-N)_n}
\left(\frac{1-p}{p}\right)^n \delta_{m,n},
\]
the following identities follow with proofs given as above, which we omit.
\begin{cor} \label{thm:33}
Let $p, q \in {{\mathbb C}_{0}}$, $M,N\in \mathbb N_0$, $N\le M$, $t\in{\mathbb C}$. Then
\begin{equation*}
\label{5:62}
\sum_{x=0}^M \binom M x q^x(1-q)^{M-x}
\left[e^{t}\hyp11{-x}{-N}{-\frac{t}{p}}\right]_N K_n(x;q,M)
= 
\left(\frac{t(q-1)}{p}\right)^n
\frac{1}{(-N)_n}
\,
\left[e^t\hyp11{n-M}{n-N}{\frac{-tq}{p}}\right]_{N-n}.
\end{equation*}
\end{cor}
\begin{cor} \label{thm:35}
Let $\gamma\in{\mathbb C}$, $p, q, \in {{\mathbb C}_{0}}$, $M,N\in \mathbb N_0$, 
$N\le M$, $t\in{\mathbb C}$, $|t|<1$. Then
\begin{eqnarray*}
\label{5:64}
&&\hspace{-1.6cm}\sum_{x=0}^M \binom M x q^x(1-q)^{M-x}
\left[(1-t)^{-\gamma}\hyp21{\gamma,-x}{-N}
{\frac{t}{p(t-1)}}\right]_N  K_n(x;q,M) 
\\
&& \hspace{4cm}=\frac{(\gamma)_n}{(-N)_n}
\left(\frac{t(q-1)}{p}\right)^n
\left[
(1-t)^{-\gamma-n}\hyp21{\gamma+n,n-M}{n-N}{\frac{-qt}{p(1-t)}}
\right]_{N-n}.
\end{eqnarray*}
\end{cor}

\subsection*{Competing interests}

The authors declare that they have no competing interests.

\subsection*{Author's contributions}

All authors completed the paper together. All authors read and approve the final manuscript.

\subsection*{Acknowledgements}

Part of this work was conducted while R.~S.~Costas-Santos was a Foriegn Guest 
Researcher in the Applied and Computational Mathematics Division at the
National Institute of Standards and Technology, Gaithersburg, Maryland, U.S.A.
The author R. S. Costas-Santos acknowledges financial support by Direcci\'on 
General de Investigaci\'on, Ministerio de Econom\'ia y Competitividad of Spain, 
grant MTM2012-36732-C03-01.


\def\cprime{$'$} \def\dbar{\leavevmode\hbox to 0pt{\hskip.2ex \accent"16\hss}d}

\end{document}